\documentclass[12pt]{amsart}

\usepackage{amssymb}

\usepackage[T1]{fontenc}
\usepackage[lining]{ebgaramond}
\usepackage[libertine]{newtxmath}

\usepackage{mathrsfs}
\usepackage{tikz-cd}

\usepackage[all]{xy}
\usepackage[headings]{fullpage}
\usepackage{paralist}

\usepackage[colorlinks,pagebackref=true]{hyperref}

\makeatletter

\@addtoreset{equation}{section}
\def\theequation{\thesection.\@arabic \c@equation}

\def\@citecolor{cyan}
\def\@linkcolor{cyan}
\def\@urlcolor{cyan}

\makeatother

\theoremstyle{plain}
\newtheorem{theorem}[equation]{Theorem}
\newtheorem{lemma}[equation]{Lemma}
\newtheorem{corollary}[equation]{Corollary}
\newtheorem{proposition}[equation]{Proposition}

\theoremstyle{definition}
\newtheorem{question}[equation]{Question}

\newtheorem{remark}[equation]{Remark}
\newenvironment{remarkbox}[1][]{%
    \begin{remark}[#1] \pushQED{\qed}}{\popQED \end{remark}}

\newtheorem{example}[equation]{Example}

\newtheorem{definition}[equation]{Definition}

\newtheorem{notation}[equation]{Notation}

\newtheorem{discussion}[equation]{Discussion}

\newtheorem{observation}[equation]{Observation}

\newtheorem{construction}[equation]{Construction}

\newcommand{\calA}{\mathcal A}
\newcommand{\bfa}{\mathbf a}
\newcommand{\fraka}{{\mathfrak a}}

\newcommand{\bfb}{\mathbf b}

\newcommand{\calD}{\mathcal D}

\newcommand{\calF}{\mathcal F}

\newcommand{\calG}{\mathcal G}

\newcommand{\bfh}{\mathbf h}

\newcommand{\bfi}{\mathbf i}

\newcommand{\calL}{\mathcal L}

\newcommand{\frakm}{{\mathfrak m}}

\newcommand{\frakn}{{\mathfrak n}}

 \let\strSh\scrO

\newcommand{\frakp}{{\mathfrak p}}

\newcommand{\naturals}{\mathbb{N}}
\newcommand{\ints}{\mathbb{Z}}

\def\to{\longrightarrow}
\DeclareMathOperator{\affine}{\mathbb{A}}
\DeclareMathOperator{\projective}{\mathbb{P}}
\DeclareMathOperator{\height}{ht}
\DeclareMathOperator{\codim}{codim}
\DeclareMathOperator{\Hom}{Hom}
\DeclareMathOperator{\Tor}{Tor}

\DeclareMathOperator{\Frac}{Frac}
\DeclareMathOperator{\Spec}{Spec}
\DeclareMathOperator{\Proj}{Proj}

\DeclareMathOperator{\homology}{H}
\newcommand{\minus}{\ensuremath{\smallsetminus}}
\DeclareMathOperator{\image}{Im}
\DeclareMathOperator{\Ext}{Ext}
\DeclareMathOperator{\Tr}{Tr}
\newcommand{\sheafHom}{\mathcal{H}om}
\newcommand{\sheafExt}{\mathcal{E}xt}
\DeclareMathOperator{\Tot}{Tot}
\makeatletter
\def\RDerChar{\mathbf{R}\!}
\def\RDer{\@ifnextchar[{\R@Der}{\ensuremath{\RDerChar}}}
\def\R@Der[#1]{\ensuremath{\RDerChar^{#1}}}
\makeatother
\newcommand{\GL}{\mathrm{GL}}

\title{Generators of top cohomology}
\author{Manoj Kummini}
\address{Chennai Mathematical Institute, Siruseri, Tamilnadu 603103. India}
\email{mkummini@cmi.ac.in}

\author{Mohit Upmanyu}
\address{Chennai Mathematical Institute, Siruseri, Tamilnadu 603103. India}
\email{mohit@cmi.ac.in}

\begin{document}

\begin{abstract}
Let $R$ be a commutative noetherian ring and $f: X \to \Spec R$ a proper
smooth morphism, of relative dimension $n$.
From Hartshorne, \emph{Residues and Duality}, Springer, 1966, one knows
that the trace map
$\Tr_f : \homology^n(X, \omega_{X/R}) \to R$
is an isomorphism when $f$ has geometrically connected
fibres. 
We construct an exact sequence that generates 
$\Ext_X^n(\strSh_X, \omega_{X/R}) = \homology^n(X, \omega_{X/R})$ as an
$R$-module in the following cases:

\begin{enumerate}

\item
when $R$ is a DVR and $f$ has a section;

\item
when $R=\ints$ and $X$ is the Grassmannian $G_{2,m}$ for some $m \geq 4$.
\end{enumerate}
This partially answers a question raised by Lipman.
\end{abstract}
\maketitle

\section{Introduction}

J.~Lipman posed the following question (in private communication to the
first author).
Let $R$ be a (commutative) noetherian ring and $f: X \to \Spec R$ a proper
smooth morphism, of relative dimension $n$.
Write $\omega_{X/R} := \bigwedge^n \Omega_{X/R} = f^!R[-n]$, the relative
dualizing sheaf.
By the duality theorem for proper 
morphisms~\cite[Chapter VII, Theorem~3.3]{HartRD66}
we get an isomorphism
\[
\RDer f_* \RDer  \sheafHom_X(\omega_{X/R}, \omega_{X/R}[n])
= \RDer \Hom_R(\RDer f_*\omega_{X/R}, R).
\]
Since $f$ is smooth, we also have a trace 
map~\cite[Chapter VII, Theorem~4.1]{HartRD66}
\[
\Tr_f : \homology^n(X, \omega_{X/R}) \to R
\]
which is an isomorphism when $f$ has geometrically connected
fibres~\cite[Chapter III, Theorem~11.2(g)]{HartRD66}.

\begin{question}[Lipman]
\label{question:lipman}
Can $\Tr_f^{-1}(1_R)$ be represented by a canonical exact sequence 
\[
0 \to \omega_{X/R} \to F_{n} \to \cdots \to F_{1} \to \strSh_X \to 0
\in 
\Ext_X^n(\strSh_X, \omega_{X/R}) = 
\homology^n(X, \omega_{X/R})?
\]
\end{question}

For example, if $X = \projective^n_R$, then the (exact) Koszul complex
\[
0 \to \omega_{X/R} \to \strSh_X(-n)^{n+1} \to \cdots \to 
\strSh_X(-1)^{n+1} \to \strSh_X \to 0
\]
generates $\Ext_X^n(\strSh_X, \omega_{X/R})$.
One obtains this Koszul complex as follows.
We have the Euler exact sequence (with $X = \projective^n_R$)
\[
0 \to \Omega_{X/R} \to 
\strSh_X(-1)^{n+1} \to \strSh_X \to 0
\]
This corresponds to a global section $s$ of 
$\strSh_X(+1)^{n+1}$, which defines a chain complex on $\bigwedge^* 
\strSh_X(-1)^{n+1}$. The exact Koszul complex given above is obtained when
we take $s$ is given by $(x_0, \ldots, x_n)$ where the $x_i$ are 
the homogeneous coordinates of $\projective^n_R$.
See~\cite[Theorem~17.5]{eiscommalg} for the relation between the cycles in
the Koszul complex and the exterior powers of 
$\Omega_{X/R}$.

In general when $X$ is a smooth connected 
projective scheme over $R$ with a closed
embedding $i : X \to \projective^N_R$, then we could consider the pull-back
of the above Euler sequence on $\projective^N_R$.
\[
\xymatrix{
0 \ar[r] & i^* \Omega_{\projective^N_R/R} \ar[r] \ar@{>>}[d]& 
i^* \strSh_{\projective^N_R}(-1)^{N+1}  \ar[r]^-{i^* s} \ar@{>>}[d]& 
i^* \strSh_{\projective^N_R}   \ar[r] \ar@{=}[d]& 0
\\
0 \ar[r] & \Omega_{X/R} \ar[r] &  \calF \ar[r]^-s &  \strSh_X \ar[r] &  0
}
\]
Here $\calF$ is the push-out and we denote the resulting map $\calF \to
\strSh_X$ also by $s$.
It is natural to consider whether the Koszul complex on $X$ given by $s$ is
a generator of 
$\Ext_X^n(\strSh_X, \omega_{X/R})$.
In Section~\ref{section:EulerKoszul}, we show that this approach fails
when $R = \ints$ and $X =G_{2,4}$. 
We show, more generally, that no element of 
$\Ext_X^1(\strSh_X, \Omega_{X/\ints})$ will give a generator of 
$\Ext_X^n(\strSh_X, \omega_{X/\ints})$.

However, in this paper, we directly describe a generator of 
$\Ext_X^n(\strSh_X, \omega_{X/R})$, in the following cases.

\begin{theorem}
\label{theorem:dvr}
Question~\ref{question:lipman} has a positive answer if $R$ is a DVR 
and $f : X \to \Spec R$ is a smooth projective morphism, with a section
$\Spec R \to X$.
\end{theorem}

\begin{theorem}
\label{theorem:grass}
Question~\ref{question:lipman} has a positive answer when $R = \ints$ and
$X$ is the Grassmannian $G_{2,m}$, $m \geq 4$.
\end{theorem}

In the next section, we outline our strategy. 
Theorem~\ref{theorem:dvr} is proved as Theorem~\ref{theorem:dvrDetails} in
Section~\ref{section:dvr}. Prior to that, in Section~\ref{section:fld}, we
prove the analogous statement (Theorem~\ref{theorem:fld})
over fields, which is used in the proof of
Theorem~\ref{theorem:dvrDetails}.
Theorem~\ref{theorem:grass} is proved in Section~\ref{section:grass}.

\subsection*{A stronger question}
After seeing a pre-print of this, Lipman suggested the following sharper
version of the question.

\begin{question}[Lipman]
\label{question:lipman2}
Is there a canonical exact sequence
\[
0 \to \omega_{X/R} \to F_{n} \to \cdots \to F_{1} \to \strSh_X \to 0
\in 
\Ext_X^n(\strSh_X, \omega_{X/R}) = 
\homology^n(X, \omega_{X/R})
\]
that is taken to $1_R$ (rather than to some unit in $R$)
under the canonical trace map defined via local residues
(\textit{cf}.~\cite[Theorem~(0.6)] {LipmanAsterisque1984})?
\end{question}

Unfortunately, our answers (Theorems~\ref{theorem:dvr}
and~\ref{theorem:grass}) to Question~\ref{question:lipman} do not answer
the stronger question.

\subsection*{Acknowledgements}

We thank Joseph Lipman for asking this question and providing us with some
observations, Patrick Polo and Steven Sam for some clarification on the
Grassmannian and Pramath Sastry for various discussions. Jerzy Weyman asked
us whether the argument for $G_{2,m}$ works for other Grassmannians;
see Remark~\ref{remarkbox:Gdn}. We also thank the referee for comments that
improved the exposition.
The computer algebra system~\cite{M2}
provided valuable assistance in studying examples.

\section{Strategy}
\label{section:strategy}

Our starting point for finding a generator for $\Ext$ is the
the following lemma.
It might be well-known, but we include a proof at the end of this section
for the sake of completeness.

\begin{lemma}
\label{lemma:KoszulGen}
Let $A$ be a Gorenstein local ring and $\bfa :=a_1, \ldots, a_m$ a regular
sequence in $A$. Then the augmented Koszul complex
\[
0 \to A \to A^{m} \to \cdots \to A^m \to A \to A/\bfa A \to 0
\]
is a generator of the $A/\bfa A$-module $\Ext_A^m(A/\bfa A, A) \simeq
A/\bfa A$.
\end{lemma}

To use the above lemma in our context, we look for a complete intersection
inside $X$ that is of relative dimension zero over $R$. For the strategy to
work, this is not enough; we also need that $X \to \Spec R$ has a section.

\begin{notation}
\label{notation:main}
Let $R$ be a Cohen-Macaulay ring with a dualizing module $\omega_R = R$
and $X$ a smooth projective $R$-scheme, of relative
dimension $n$. Assume that
\begin{enumerate}

\item
the map $X \to \Spec R$ has a section $Z$;

\item
There exists a rank-$n$ locally free sheaf $\calF$ on $X$ and $\sigma \in
\homology^0(X, \calF)$ such that the zero section $X_n$ of 
$\sigma$ is flat and of relative dimension $0$ over $R$;

\item
$Z$ is a connected component of $X_n$.
\end{enumerate}
\end{notation}

Note that if $R$ is not local, the hypothesis on $R$ above
might be stronger than requiring $R$ to be Gorenstein, i.e., $R_\frakp$ is
a Gorenstein local ring for all $\frakp \in \Spec R$.

\begin{notation}
We consider only relative dualizing sheaves in this paper. Therefore we
abbreviate $\Omega_{X/R}$ as $\Omega_X$
and $\omega_{X/R}$ as $\omega_X$.
\end{notation}

Given such an $\calF$ and $\sigma$, we have the Koszul complex
\[
0 \to \det \calF^* \to \bigwedge^{n-1} \calF^*
\to \cdots \to 
\bigwedge^{2} \calF^* \to
\calF^* \stackrel{\sigma}\to
\strSh_X \to 0,
\]
which is a locally free resolution of $\strSh_{X_n}$,
since $X_n$ has codimension $n$ inside $X$.
Apply $\sheafHom_{\strSh_X}(-,\omega_X)$ and note that 
$\omega_{X_n} = \sheafExt^n_{\strSh_X}(\strSh_{X_n}, \omega_X)$
to obtain an exact sequence
\[
0 \to \omega_X  \to \omega_X \otimes \calF
\to \cdots \to \omega_X \otimes \det \calF
\to \omega_{X_n} \to 0.
\]
Since $X_n$ is local complete intersection inside $X$ and $\omega_X$ is
a locally free $\strSh_X$-module of rank $1$, we see that 
$\omega_{X_n}$ is a locally free $\strSh_{X_n}$-module of rank $1$.
We write $\widetilde K_\bullet$ for this exact sequence.
It is labelled in such a way that $\widetilde{K}_n = \omega_X \otimes \calF$ 
and 
$\widetilde{K}_1 = \omega_X \otimes \det \calF$.

Since $Z$ is a connected component of $X_n$, we can write
$X_n$ as a disjoint union $Z \bigsqcup Z'$. Then we have a split exact
sequence
\begin{equation}
\label{equation:OXnsplits}
\begin{tikzcd}
    0 \arrow{r} & \strSh_{Z'} \arrow{r} & \strSh_{X_n} \arrow{r}{\pi} & \strSh_Z \arrow{r}\arrow[bend left=33]{l}{i} & 0
\end{tikzcd}
\end{equation}
of $\strSh_X$-modules. Similarly, we also have a split exact sequence
\begin{equation}
\label{equation:OmegaXnsplits}
\begin{tikzcd}
    0 \arrow{r} & \omega_{Z'} \arrow{r} & \omega_{X_n} \arrow{r} & \omega_Z \arrow{r}\arrow[bend left=33]{l}{i_{\omega}} & 0.
\end{tikzcd}
\end{equation}
By assumption, there is a canonical isomorphism from $\strSh_Z$ to $\omega_Z$.
(The map $f : Z \to \Spec R$ is identity. Hence $\omega_Z = f^! \strSh_Z$
is canonically isomorphic to $\strSh_Z$.)
By abuse of notation, we will also use $i_{\omega} : \strSh_Z \to \omega_{X_n}$ to denote the composite of this isomorphism and $i$.

Now consider the following commutative diagram. The lowest row is the exact
sequence $\widetilde{K}_\bullet$.
The middle row is the pull-back of the lowest row along the map $i_\omega$.
The top row is the pull-back of the middle row along the natural surjective map $\strSh_X
\to \strSh_Z$.
\begin{equation}
\label{equation:pullbackOfExt}
\begin{gathered}
\xymatrix{
0 \ar[r] & \omega_X \ar[r] \ar@{=}[d]& \cdots \ar[r] & 
\widetilde{K}_2 \ar[r] \ar@{=}[d]&
\calG_1 \ar[r] \ar[d]&
\strSh_{X} \ar[r] \ar[d] & 0
\\
0 \ar[r] & \omega_X \ar[r] \ar@{=}[d]& \cdots \ar[r] & 
\widetilde{K}_2 \ar[r] \ar@{=}[d]&
\calG_0 \ar[r] \ar[d]&
\strSh_{Z} \ar[r] \ar[d]^{i_\omega}& 0
\\
0 \ar[r] & \omega_X \ar[r] & \cdots \ar[r] & 
\widetilde{K}_2 \ar[r] &
\widetilde{K}_1\ar[r] &
\omega_{X_n} \ar[r] & 0
}
\end{gathered}
\end{equation}
Writing $I_{Z'}$ for the ideal
sheaf of $Z'$ in $X$, we see that
$\calG_0 = I_{Z'} \otimes \tilde K_1$.

The goal is to show that the top row of the above diagram 
generates $\Ext_X^n(\strSh_X, \omega_X)$ as an $R$-module.
In the next proposition, we show that the middle row 
in the diagram~\eqref{equation:pullbackOfExt} generates 
$\Ext_X^n(\strSh_Z, \omega_X)$ as an $R$-module.

\begin{proposition}
\label{proposition:midrow}
Let $R$ and $X$ be as above.
Then

\begin{enumerate}

\item
\label{proposition:midrow:stalkGen}
Let $x \in X_n$. Then the stalk of 
$\widetilde K_\bullet$ at $x$ 
generates the stalk of 
the sheaf $\sheafExt^n_{\strSh_X}(\strSh_{X_n}, \omega_X)$ at $x$
as an $\strSh_{X, x}$-module.

\item
\label{proposition:midrow:midrow}
The middle row in the diagram~\eqref{equation:pullbackOfExt} generates 
$\Ext_X^n(\strSh_Z, \omega_X)$ as an $R$-module.

\end{enumerate}
\end{proposition}

\begin{proof}
\eqref{proposition:midrow:stalkGen}:
Note that $X$ and, therefore, $X_n$ are Gorenstein.
The proposition now follows from Lemma~\ref{lemma:KoszulGen}.

\eqref{proposition:midrow:midrow}
From~\eqref{equation:OXnsplits} we see that $i^*\pi^*$
is the identity map on $\sheafExt^n_{\strSh_X}(\strSh_{Z}, \omega_X)$.
Hence, by~\eqref{proposition:midrow:stalkGen}, the middle row generates
the stalk of $\sheafExt^n_{\strSh_X}(\strSh_{Z}, \omega_X)$ at every point
on $Z$. Note that, since $X$ is locally Cohen-Macaulay and
$\codim_X(Z) = n$, 
$\sheafExt^i_{\strSh_X}(\strSh_{Z}, \omega_X)=0$ for all $i<n$.
Therefore 
$\Ext_X^n(\strSh_Z, \omega_X) 
= \Gamma(X,\sheafExt^n_{\strSh_X}(\strSh_{Z}, \omega_X))$.
Since $Z = \Spec R$, we get the proposition.
\end{proof}

We want to show that the top row in the
diagram~\eqref{equation:pullbackOfExt}
generates $\Ext_X^n(\strSh_X, \omega_X)$ as an $R$-module
(with suitable further hypothesis on $R$), by showing that
the natural map
$\Ext_X^n(\strSh_Z, \omega_X) \to
\Ext_X^n(\strSh_X, \omega_X)$
is an isomorphism.
Write $I_Z$ for the ideal sheaf of $Z$ in $X$.
Then we have the exact sequence
\begin{equation}
\label{equation:fiveTermOnExt}
\Ext_X^{n}(\strSh_Z, \omega_X) \to 
\Ext_X^{n}(\strSh_X, \omega_X) \to 
\Ext_X^{n}(I_Z, \omega_X).
\end{equation}
By Proposition~\ref{proposition:midrow}\eqref{proposition:midrow:midrow},
$\Ext^{n}(\strSh_Z, \omega_X)$ is a cyclic $R$-module.
Hence it would suffice to show that
\begin{equation}
\label{equation:ExtOXExtIX}
\Ext_X^{n}(\strSh_X, \omega_X) = R \;\text{and}\; 
\Ext_X^{n}(I_Z, \omega_X) = 0.
\end{equation}
We start with a few preliminary results.

\begin{lemma}
\label{lemma:HZeroWithSecFld}
Let $K$ be a field and $Y$ a connected reduced projective $K$-scheme with a
$K$-rational point. Then the natural map $K \to \homology^0(Y, \strSh_Y ) $
is an isomorphism.
\end{lemma}

\begin{proof}
Since $Y$ is a connected reduced projective $K$-scheme, 
$\homology^0(Y, \strSh_Y )$ is a local reduced ring, finite over $K$, i.e,
a finite extension field of $K$. Since $Y$ has a $K$-rational point, there
is a map $\homology^0(Y, \strSh_Y ) \to K$ of $K$-algebras. This proves the
lemma.
\end{proof}

\begin{proposition}
\label{proposition:HZeroWithSec}
Assume that $R$ is a normal domain and that $X$ is connected and reduced.
Then the natural maps $R \to \homology^0(X, \strSh_X)$ 
and
$\homology^0(X, \strSh_X ) \to \homology^0(Z, \strSh_Z)$
are isomorphisms.
\end{proposition}

\begin{proof}
Write $K = \Frac(R)$.
We claim that $(R \to \homology^0(X, \strSh_X)) \otimes_R K$ is an
isomorphism. Assume the claim. Note that 
$\homology^0(X, \strSh_X)$ is an integral domain that is finite over
$R$. By the claim, $\Frac(\homology^0(X, \strSh_X) ) = K$. Since $R$ is
normal, we get the first statement in proposition. The second statement
follows immediately.
To prove the claim, write $X_K$ for the generic fibre $X \times_{\Spec R}
\Spec K$. It is an integral scheme and has a $K$-rational point $Z_K$.
By Lemma~\ref{lemma:HZeroWithSecFld}, the natural map $K \to 
\homology^0(X_K, \strSh_{X_K})$ is an isomorphism.
By flat base-change, 
$\homology^0(X_K, \strSh_{X_K}) = \homology^0(X, \strSh_X) \otimes_R K$.
This proves the claim.
\end{proof}

We state the next proposition for one-dimensional PIDs, because the proofs
of the ensuing lemmas assume that the ring is not a field. However the
statement is true over fields also. In fact, over a field, we can prove the
existence of the scheme $X_n$ in Notation~\ref{notation:main} with the
desired properties. This is done in the next section.

\begin{proposition}
\label{proposition:toprowPID}
Assume that $R$ is a PID that is not a field. 
Then the top row in the
diagram~\eqref{equation:pullbackOfExt}
generates $\Ext_X^n(\strSh_X, \omega_X)$ as an $R$-module.
\end{proposition}

We prove this with a series of lemmas.

\begin{lemma}
\label{lemma:strShXfree}
Assume that $R$ is a PID that is not a field. Then
$\homology^1(X, \strSh_X)$ is a free $R$-module.
\end{lemma}

\begin{proof}
Let $a$ be an irreducible element of $R$ and $\Bbbk = R/(a)$. 
Consider the exact sequence
\[
0 \to 
\homology^0(X, \strSh_X) \stackrel{\cdot a} \to 
\homology^0(X, \strSh_X) \to
\homology^0(X_{\Bbbk}, \strSh_{X_{\Bbbk}})
\to
\homology^1(X, \strSh_X) \stackrel{\cdot a} \to 
\homology^1(X, \strSh_X) \to.
\]
where $X_\Bbbk = X \times_{\Spec R} \Spec \Bbbk$.
Then $Z \times_{\Spec R} \Spec \Bbbk$ is 
a $\Bbbk$-rational point of the $\Bbbk$-scheme $X_{\Bbbk}$.
Hence 
$\homology^0(X_{\Bbbk}, \strSh_{X_{\Bbbk}}) = \Bbbk$ by
Lemma~\ref{lemma:HZeroWithSecFld}.
By Proposition~\ref{proposition:HZeroWithSec},
$\homology^0(X, \strSh_X) = R$.
Therefore the map 
$\homology^1(X, \strSh_X) \stackrel{\cdot a} \to 
\homology^1(X, \strSh_X)$
is injective.
Since $a$ is arbitrary, 
$\homology^1(X, \strSh_X)$ is torsion-free and, hence, free.
\end{proof}

\begin{lemma}
\label{lemma:IZfree}
Assume that $R$ is a PID that is not a field. Then
$\homology^0(X, I_Z) = 0$ and $\homology^1(X, I_Z)$ is a free $R$-module.
\end{lemma}

\begin{proof}
Consider the exact sequence
\[
0 \to 
\homology^0(X, I_Z)  \to 
\homology^0(X, \strSh_X) \stackrel{p} \to
\homology^0(Z, \strSh_Z) \to
\homology^1(X, I_Z)  \to 
\homology^1(X, \strSh_X) \to.
\]
The map $p$ is an isomorphism, since
$\homology^0(X, \strSh_X) \simeq R \simeq \homology^0(Z, \strSh_Z)$
and $p(1) = 1$.
Hence, firstly, $\homology^0(X, I_Z) = 0$.
Now note that $\homology^1(X, I_Z)$ is torsion-free, since it is a
submodule of the free module 
$\homology^1(X, \strSh_X)$ (Lemma~\ref{lemma:strShXfree}).
Therefore $\homology^1(X, I_Z)$ is free.
\end{proof}

\begin{lemma}
\label{lemma:UnivCoeff}
Assume that $R$ is a PID that is not a field.
Let $F$ be a quasi-coherent sheaf on $X$ and $q$ an integer.
Then there is an exact sequence
\[
0 \to \Ext_R^1(\homology^{q+1}(X, F ), R ) \to 
\Ext^{n-q}(F, \omega_X) \to
\Hom_R(\homology^q(X, F ), R ) \to  0
\]
of $R$-modules.
\end{lemma}

\begin{proof}
By $\calD_R^\bullet$ for the normalized dualizing complex of $R$.
Note that $\calD_R^\bullet$ is isomorphic to the complex $R[1]$ (i.e., the
module $R$ sitting in cohomological index $-1$) in the derived category of
$R$. Write $f$ for the map $X \to \Spec R$. Then
\[
f^! \calD_R^\bullet := f^*\calD_R^\bullet \otimes_{\strSh_X} \bigwedge^n
\Omega_X [n] = \omega_X[n+1]
\]
is a dualizing complex for $X$.
By duality for proper morphisms, we have an isomorphism 
\[
\RDer f_* \RDer \sheafHom_X(F, f^! \calD_R^\bullet) = 
\RDer \sheafHom_R(\RDer f_* F, \calD_R^\bullet)
\]
in the derived category of $R$.
Since $\RDer f_* = \RDer \Gamma$, we get 
$\RDer \Hom_X(F, f^! \calD_R^\bullet) = 
\RDer \Hom_R(\RDer \Gamma F, \calD_R^\bullet)$.
Note that 
\[
\Ext^{n-q}(F, \omega_X) = 
\homology^{n-q}(\RDer \Hom_X(F, f^! \calD_R^\bullet[-n-1]))
=
\homology^{-q-1}(\RDer \Hom_X(F, f^! \calD_R^\bullet)).
\]
We therefore compute 
$\homology^{-q-1}(\RDer \Hom_R(\RDer \Gamma F, \calD_R^\bullet))$ using a
spectral sequence.
Fix a bounded-below complex $G^\bullet$ of $R$-modules that represents 
$\RDer \Gamma F$.
Without loss of generality, we can assume that $ \calD_R^\bullet$ is an 
injective resolution of $R[1]$. Hence we have a double complex
\[
E_0^{-i,-j} = \Hom_R(G^i,\calD_R^{-j})
\]
which is $0$ whenever $-j \not \in \{0, -1 \}$.
Taking the filtration by rows, we get a spectral sequence
${}_hE_1^{-i,-j} = \Hom_R(\homology^i(X, F),\calD_R^{-j})$
and, then
\[
{}_hE_\infty^{-i,-j} = 
{}_hE_2^{-i,-j} = 
\begin{cases}
\Ext_R^1(\homology^i(X, F), R), & -j=0;\\
\Hom_R(\homology^i(X, F), R), & -j=-1; \\
0, & \text{otherwise}.
\end{cases}
\]
The edge map is ${}_hE_\infty^{-i,0}  \to 
\homology^{-i}(\RDer \Hom_R(\RDer \Gamma F, \calD_R^\bullet))$; this is
injective, and the cokernel is 
${}_hE_\infty^{-i+1,-1}$.
Thus we get an exact sequence
\[
0 \to \Ext_R^1(\homology^{q+1}(X, F), R)
\to \Ext^{n-q}(F, \omega_X) \to 
\Hom_R(\homology^q(X, F), R) \to 0
\]
of $R$-modules.
\end{proof}

\begin{corollary}
\label{corollary:Extn}
With notation as above,
\begin{enumerate}
\item
\label{corollary:Extn:strSh}
$\Ext^{n}(\strSh_X, \omega_X) \simeq R$ as $R$-modules.

\item
\label{corollary:Extn:IZ}
$\Ext^{n}(I_Z, \omega_X) = 0$.
\end{enumerate}

\end{corollary}

\begin{proof}

\eqref{corollary:Extn:strSh}:
Apply Lemma~\ref{lemma:UnivCoeff} with $q=0$ and $F=\strSh_X$.
By Lemma~\ref{lemma:strShXfree}, 
$\Ext_R^{1}(\homology^1(X, \strSh_X ), R ) = 0$, so
$\Ext^{n}(\strSh_X, \omega_X) \simeq \Hom_R(\homology^0(X, \strSh_X),R)
\simeq R$.

\eqref{corollary:Extn:IZ}:
Apply Lemma~\ref{lemma:UnivCoeff} with $q=0$ and $F=I_Z$.
Now use Lemma~\ref{lemma:IZfree}.
\end{proof}

We can now prove Proposition~\ref{proposition:toprowPID}.
\begin{proof}[Proof of Proposition~\protect{\ref{proposition:toprowPID}}]
Corollary~\ref{corollary:Extn} establishes~\eqref{equation:ExtOXExtIX}.
Hence the proposition follows from~\eqref{equation:fiveTermOnExt}.
\end{proof}

We conclude this section by giving a proof of 
Lemma~\ref{lemma:KoszulGen}.

\begin{proof}[Proof of Lemma~\protect{\ref{lemma:KoszulGen}}]
We prove this by induction on $m$. When $m=1$, we have 
\begin{equation}
\label{equation:oneTermKosz}
0 \to A \stackrel{\cdot a_1}\to A \to A/a_1A \to 0,
\end{equation}
which, when applied $\Hom_A(-,A )$, gives
\[
0 \to \Hom_A(A,A ) \to \Hom_A(A,A ) \to \Ext_A^1(A/a_1A,A ) \to 0.
\]
Then~\eqref{equation:oneTermKosz} corresponds to the image of
$\mathrm{id}_A$ inside $\Ext_A^1(A/a_1A,A )$ 
(see, e.g.,~\cite[XIV, Theorem~1.1]{CartanEilenberg1999})
which is a generator of $\Ext_A^1(A/a_1A,A )$ since 
$\Hom_A(A,A )$ is a cyclic $A$-module generated by 
$\mathrm{id}_A$.

Now assume that $m>1$ and that the assertion holds for the sequence 
$\bfb := a_1, \ldots, a_{m-1}$.
The Koszul complex $K_\bullet(\bfa)$ (without the augmentation map
$K_0(\bfa) \to A/\bfa A$) is the mapping cone of the map 
\[
K_\bullet(\bfb)\stackrel{\cdot a_m}\to K_\bullet(\bfb)
\]
of complexes.
Then we have an exact sequence 
\[
0 \to K_\bullet(\bfb) \to K_\bullet(\bfa) \to
K_\bullet(\bfb)[-1] \to 0
\]
of complexes.
An element $\xi \in \Ext^{m-1}_A(A/\bfb A, A )$ is given by a map
$K_\bullet(\bfb) \to A[-m+1]$ in the derived category of $A$-modules.
(Here $A$ is thought of as the complex $0 \to A \to 0$, with $A$ in
homological position $0$.)
The image of $\xi$ under the connecting morphism 
$\Ext^{m-1}_A(A/\bfb A, A ) \to \Ext^{m}_A(A/\bfa A, A )$
is given by the composite map
\[
K_\bullet(\bfa) \to K_\bullet(\bfb)[-1] \to A[-m].
\]
In particular, if we take the identity map $K_{m-1}(\bfb) \to A$ (which
gives a generator of $\Ext^{m-1}_A(A/\bfb A, A )$), we get the identity
map $K_{m}(\bfa) \to A$ (which
gives a generator of $\Ext^{m}_A(A/\bfa A, A )$).
\end{proof}

\section{Example: the projective plane}

In this section, we exhibit 
$\projective^2_\ints$
as an example of the strategy described in the
previous section. We can get the
middle row of~\eqref{equation:pullbackOfExt} directly
(for a suitable choice of $X_2$ and section $Z$).
We will then show that the top row of~\eqref{equation:pullbackOfExt} 
equals the Koszul complex in $\Ext^2(\strSh_X,\omega_X)$.

Write $X = \projective^2_\ints$. 
Let $x_0, x_1, x_2$ be homogeneous coordinates on $X$ and $X_2$ be the
subscheme defined by $(x_1, x_2)$. Take $Z = X_2$.
In other words, $\calF = \strSh_X(1)^{\oplus 2} = 
\strSh_X(1) \epsilon_1 \oplus \strSh_X(1) \epsilon_2$ 
and $\sigma = (x_1, x_2)
= x_1 \epsilon_1 + x_2 \epsilon_2$.
Then
$\omega_X = \strSh_X(-3) = 
\strSh_X (-3)\cdot (\epsilon_1^* \wedge \epsilon_2^* \wedge \epsilon_3^*)$.
Since $X_2 = \Spec \ints$, $\omega_{X_2}=\strSh_{X_2}$.
Abbreviate $e_i \wedge e_j = e_{ij}$ etc.
Hence the bottom and the middle rows of~\eqref{equation:pullbackOfExt} are
equal to
\[
\xymatrix@C=1em{
0 \ar[r] & \strSh_X(-3 )\epsilon_{123}^*
\ar[rrr]^{ \begin{bmatrix} x_1 \\ -x_2 \end{bmatrix} } &&&
{\begin{matrix}
\strSh_X(-2 )\epsilon_{23}^* \\
\oplus \\
\strSh_X(-2)\epsilon_{13}
\end{matrix}}
\ar[rrr]^{ \begin{bmatrix} x_2 & x_1 \end{bmatrix} } &&&
\strSh_X(-1)\epsilon_3^*
\ar[r] &
\strSh_Z\epsilon_3^*
\ar[r] & 0
}
\]
The Koszul complex on the section $\sum_{i=1}^3 x_i \epsilon_i^*$ of 
$\oplus_{i=1}^3 \strSh_X \epsilon_i^* = \strSh_X(-1)^{\oplus 3}$ is
\[
\xymatrix@C=1em{
0 \ar[r] & \strSh_X(-3)\epsilon_{123}^*
\ar[rr]^{ \begin{bmatrix} x_1 \\ -x_2 \\ x_3\end{bmatrix} } &&
{\begin{matrix}
\strSh_X(-2)\epsilon_{23}^* \\
\oplus \\
\strSh_X(-2)\epsilon_{13}^* \\
\oplus \\
\strSh_X(-2)\epsilon_{12}^*
\end{matrix}}
\ar[rrrrr]^{ \begin{bmatrix} 
x_2 & x_1 & 0 \\
- x_3 & 0 & x_1 \\
0 & - x_3 & -x_2 \\
\end{bmatrix} } &&&&&
{
\begin{matrix}
\strSh_X(-1)\epsilon_3^* \\
\oplus \\
\strSh_X(-1)\epsilon_2^* \\
\oplus \\
\strSh_X(-1)\epsilon_1^*
\end{matrix}
}
\ar[rrrr]^{ \begin{bmatrix} x_3 & x_2 & x_1 \end{bmatrix} } &&&&
\strSh_X
\ar[r] & 0
}
\]

Therefore we get a commutative diagram; for the sake of brevity
we suppress the basis elements and the matrices.
The right-most vertical map is the canonical surjection, while the other
maps are the projection maps to the respective summands.
\begin{equation}
\label{equation:P2Koszulmap}
\xymatrix{ 
0 \ar[r] & \strSh_X(-3) \ar[r] \ar@{=}[d] &
\strSh_X(-2)^{\oplus 3} \ar[r] \ar[d]^{\mathrm{proj}} &
\strSh_X(-1)^{\oplus 3} \ar[r] \ar[d]^{\mathrm{proj}} &
\strSh_X \ar[r] \ar[d]^{\mathrm{can}} &
0
\\
0 \ar[r] & \strSh_X(-3) \ar[r] &
\strSh_X(-2)^{\oplus 2} \ar[r]  &
\strSh_X(-1) \ar[r]  &
\strSh_Z \ar[r]  &
0
}
\end{equation}

Now use~\cite[Proposition~III.5.1]{MacLHomology63} to see that the
push-forward of the top row of~\eqref{equation:P2Koszulmap}
along the map $\mathrm{id}_{\strSh_X(-3)}$ is the same as the pull-back of
the bottom row of~\eqref{equation:P2Koszulmap}
along the canonical map $\strSh_X \to \strSh_Z$.
Hence the top row of~\eqref{equation:P2Koszulmap}
is the top row of~\eqref{equation:pullbackOfExt}.

\section{Over a field}
\label{section:fld}

In this section, we show that the strategy described in the previous
section can be carried out over a field. This result will be used to prove
the result for DVR.

\begin{theorem}
\label{theorem:fld}
Let $R$ be a field.
Let $X$ be a $n$-dimensional smooth projective $R$-subscheme of 
$\projective^N_R$.
Assume that $X$ has an $R$-rational point $Z$.
Write $S = R[x_0, \ldots, x_N]$ for the homogeneous coordinate ring of
$\projective^N_R$. 
Then:

\begin{enumerate}

\item
\label{theorem:fld:existsfk}
There exist homogeneous polynomials $ f_1, \ldots,  f_n \in S$ 
such that for every $1 \leq k \leq n$, 
$ f_1, \ldots,  f_k$ define a complete intersection inside $X$ containing $Z$.
Moreover, $Z$ is an isolated point of the subscheme $X_n$ defined by 
$ f_1, \ldots,  f_n \in S$.

\item
\label{theorem:fld:extgen}
The top row of~\eqref{equation:pullbackOfExt} generates
$\Ext_X^n(\strSh_X, \omega_X)$ as an $R$-vector space.
Moreover, the base-change of this exact sequence generates
$\Ext_{X_S}^n(\strSh_{X_S}, \omega_{X_S})$
for all ring maps $R \to S$ (where $X_S = X \times_{\Spec R} \Spec S$).
\end{enumerate}
\end{theorem}

For a homogeneous ideal $J$ of 
$R[x_0, \ldots, x_N]$, we write 
$J^\mathrm{sat} = \bigcup_i \left(J : (x_0, \ldots, x_N)^i\right)$, i.e,
the ideal of the projective subscheme of $\projective_R^N$ defined by
$J$.

\begin{proof}[Proof of Theorem~\protect{\ref{theorem:fld}}]
Note that $Z$ is defined by $n$ linearly independent linear forms. After a
suitable change of coordinates, we may assume that $Z$ is defined by 
$\frakp := (x_1, \ldots, x_n)$.

\eqref{theorem:fld:existsfk}:
We prove this by induction on $k$. Let $k=1$. 
Note that $I_{X}$ is a prime ideal of height $N-n< N$.
Hence $\frakp \nsubseteq \left(\frakp^2  \cup I_{X}\right)$
and the assertion holds for $k=1$.

Now assume that $1 < k \leq n$ and that we have found $f_1, \ldots, f_{k-1}$ with
the desired properties.
Since $X$ is non-singular, 
$\bar{S}/I_{X}$ 
and, therefore,
$\bar{S}/(I_{X}+(f_1, \ldots, f_{k-1}))^{\mathrm{sat}}$
are Cohen-Macaulay except possibly at the
irrelevant ideal.
Since $\height (I_{X}+(f_1, \ldots, f_{k-1})) = N-n+k-1 < N$,
we can find an
$f_k \in \frakp \minus \frakp^2$ that is a non-zero-divisor in 
$\bar{S}/(I_{X}+(f_1, \ldots, f_{k-1}))^{\mathrm{sat}}$.

\eqref{theorem:fld:extgen}:
Write $I_Z$ for the ideal sheaf of $Z$ in $X$.
By Lemma~\ref{lemma:HZeroWithSecFld}, $\homology^0(X, \strSh_X) = R$, and
arguing as in the proof of Lemma~\ref{lemma:IZfree}, 
$\homology^0(X, I_Z) = 0$.
Hence, by Serre duality~\cite[Theorem~III.7.6]{HartAG},
\eqref{equation:ExtOXExtIX} holds; therefore by
\eqref{equation:fiveTermOnExt}, we get the first assertion.
Since $S$ is $R$-flat, we get the second assertion by flat base-change.
\end{proof}

(The statement of Serre duality in~\cite[Theorem~III.7.6]{HartAG} assumes
that the base field is algebraically closed. However, this assumption is
used only to see that local rings on projective spaces (over a field) are
regular. This holds over arbitrary fields, and not just over algebraically
closed fields.)

\section{Over a DVR}
\label{section:dvr}

In this section, we show that the strategy described in the previous
section can be carried out over a DVR. We assume that 
$X \to \Spec R$ has a section (with $R$ a DVR) has a section $Z$,
and show that there exists a complete intersection $X_n \subseteq X$ in
which $Z$ is a connected component.

\begin{theorem}
\label{theorem:dvrDetails}
Let $(R,\frakm, \Bbbk)$ be a DVR.
Let $X$ be a projective $R$-subscheme of 
$\projective^N_R$. Assume that the structure morphism $X \to \Spec R$ is
smooth and of relative dimension $n$,
and that it has a section $\Spec R \to Z \subseteq X$.
Write $S = R[x_0, \ldots, x_N]$ for the homogeneous coordinate ring of
$\projective^N_R$. 
Then:

\begin{enumerate}

\item
\label{theorem:dvrDetails:existsfk}
There exist homogeneous polynomials $\tilde f_1, \ldots, \tilde f_n \in S$ 
such that they define a complete intersection $X_n$ inside $X$.

\item
\label{theorem:dvrDetails:flatZerodim}
The natural map $X_n \to \Spec R$ is flat and has zero-dimensional fibres.

\item
\label{theorem:dvrDetails:Zcomp}
$Z$ is a connected component of $X_n$. Hence $\strSh_Z$ is a direct summand
of $\strSh_{X_n}$.

\item
\label{theorem:dvrDetails:extgen}
The top row of~\eqref{equation:pullbackOfExt} generates
$\Ext_X^n(\strSh_X, \omega_X)$ as a free $R$-module.
Moreover, the base-change of this exact sequence generates
$\Ext_{X_S}^n(\strSh_{X_S}, \omega_{X_S})$
for all ring maps $R \to S$ (where $X_S = X \times_{\Spec R} \Spec S$).
\end{enumerate}
\end{theorem}

Various parts of the above theorem will be proved as separate lemmas /
propositions, in more generality than stated in the theorem: 
\eqref{theorem:dvrDetails:existsfk}
and~\eqref{theorem:dvrDetails:flatZerodim} as
Proposition~\ref{proposition:div}; \eqref{theorem:dvrDetails:Zcomp}
as Proposition~\ref{proposition:Zcomp}; finally
\eqref{theorem:dvrDetails:extgen} is proved.
We start with an observation.

\begin{proposition}
Without loss of generality, $Z$ is defined by the ideal 
$(x_1, \ldots, x_N)$.
\end{proposition}
 
\begin{proof}
Without loss of generality, the homogeneous coordinate $x_0$ does not
vanish at the closed point of $Z$. Hence it does not vanish at the generic
point of $Z$ also. Therefore we may assume that $Z$ lies inside the open
subset $\{x_0 \neq 0\} \simeq \affine^N_R$. Write $y_i = \frac{x_i}{x_0}$.

Write $Z = \Spec(R[y_1, \ldots, y_N]/I)$, or, equivalently,
$R[y_1, \ldots, y_N]/I = R$. Let $a_i \in R$ be the image of $y_i$, $1 \leq
i \leq N$. Then $(y_1-a_1, \ldots, y_N-a_N ) \subseteq I$.
Hence $I = (y_1-a_1, \ldots, y_N-a_N )$.
Therefore $Z$ is defined inside $\projective^N_R$ by 
$(x_1-a_1x_0, \ldots, x_N-a_nx_0 )$. After a change of coordinates, we may
assume that this is the ideal $(x_1, \ldots, x_N)$.
\end{proof}

Write $s$ for the closed point $\Spec \Bbbk \in \Spec R$.
For subschemes $Y$ of $\projective^N_R$, we write $Y_{s}$ for its
closed fiber over $R$.     

\begin{proposition}
\label{proposition:divOnClFib}
Write $\bar{S} := \Bbbk[x_0, \ldots, x_N]$
and $\frakp := (x_1, \ldots, x_N)\bar S$.
There exist homogeneous 
$f_1, \ldots, f_n \in \frakp \minus \frakp^2$
such that for all $1 \leq k \leq n$, $f_k$ is a non-zero-divisor in 
$\bar{S}/(I_{X_{s}}+(f_1, \ldots, f_{k-1}))^{\mathrm{sat}}$.
\end{proposition}

\begin{proof}
See the proof of Theorem~\ref{theorem:fld}\eqref{theorem:fld:existsfk}.
\end{proof}

\begin{proposition}
\label{proposition:div}
Assume that $(R, \frakm, \Bbbk)$ is a local Cohen-Macaulay ring. 
For $1 \leq k \leq n$, write 
\[
f_k = x_0^{\deg f_k-1} l_k(x_1, \ldots, x_N) +  \text{terms of lower degree
in $x_0$},
\]
where $l_k$ is linear.
Let $\tilde l_k (x_1, \ldots, x_N)$ be a homogeneous linear lift of $l_k$
to $R[x_1, \ldots, x_N]$.
Consider any lift $\tilde f_k$ of $f_k$ to $R[x_0, \ldots, x_N]$ of the
form
\[
x_0^{\deg f_k-1} \tilde l_k(x_1, \ldots, x_N) +  \text{terms of lower degree
in $x_0$}.
\]
Let $X_k = 
X \cap D_{\tilde f_1 } \cap \cdots \cap D_{\tilde f_{k}}$.
(Here $D_{\tilde f_i }$ is the subscheme defined by ${\tilde f_1 }$ and the
intersection is the scheme-theoretic
intersection.)
Then
\begin{enumerate}

\item
\label{proposition:div:order}
$\tilde f_k \in (x_1, \ldots, x_N) \minus \left((x_1, \ldots, x_N)^2 + 
I_{X_{k-1}}\right)$.

\item
\label{proposition:div:nzd}
$X_k$ is Cohen-Macaulay and flat over $R$, of relative dimension $n-k$.

\end{enumerate}
\end{proposition}

\begin{proof}

\eqref{proposition:div:order}: 
It is immediate from the definition that $\tilde f_k \in (x_1, \ldots,
x_N)$. The other part follows from
Proposition~\ref{proposition:divOnClFib}
since the reduction of the ideal
$\left((x_1, \ldots, x_N)^2 + I_{X_{k-1}}\right)$ modulo $\frakm_A$
lies inside
$\left((x_1, \ldots, x_N)^2 + I_{(X_{k-1})_{s}}\right)$.

\eqref{proposition:div:nzd}:
$X$ is Cohen-Macaulay and flat over $R$. 
Hence $X_{s}$ is 
Cohen-Macaulay~\cite[Corollary to Theorem~23.3]{MatsCRT89}.
By Proposition~\ref{proposition:divOnClFib}, $f_1, \ldots, f_k$ is a
regular sequence on $X_{s}$.
Hence by~\cite[Corollary to Theorem~22.5]{MatsCRT89}, 
$\tilde f_1, \ldots \tilde f_k$ is a regular sequence on $X$ and $X_k$ is 
flat over $R$.  Therefore $X_k$ is additionally Cohen-Macaulay and of
relative dimension $n-k$.
\end{proof}

\begin{proposition}
\label{proposition:Zcomp}
Suppose that $(R, \frakm, \Bbbk)$ is a regular local ring.
Then $Z$ is a connected component of $X_n$.
\end{proposition}

\begin{proof}
$X_n$ is finite and flat over $R$ (Proposition~\ref{proposition:div}).
In order to prove the assertion, it suffices to show that $Z$ is the only
irreducible component containing the closed point of $Z$.
Let $(A, \frakn)$ be the local ring of $\strSh_{X_n}$ at 
the closed point of $Z$.

Since $X_n$ is flat over $R$,  $A$ is flat over $R$, so every $R$-regular
sequence in $\frakm$ is also $A$-regular.
In particular $A$ is Cohen-Macaulay.
Since $Z_{s}$ is an isolated point of 
$(X_n)_{s}$, it follows that $A/\frakm A = \Bbbk$.
Hence the multiplicity $e(\frakm A, A )$ of the ideal $\frakm A$ in $A$ is
$1$, by~\cite[Theorem~17.11]{MatsCRT89}.
Hence $e(\frakn, A) = 1$, which implies that $A$ is a regular local ring
and, \textit{a fortiori}, a domain
which is what we needed to show.
\end{proof}

By now, we have proved~\eqref{theorem:dvrDetails:existsfk},
\eqref{theorem:dvrDetails:flatZerodim} and
\eqref{theorem:dvrDetails:Zcomp} of Theorem~\ref{theorem:dvrDetails}.

\begin{proof}[Conclusion of the proof of
Theorem~\protect{\ref{theorem:dvrDetails}}]
The first assertion of
\eqref{theorem:dvrDetails:extgen} follows from
Proposition~\ref{proposition:toprowPID}.

To complete the proof of Theorem~\ref{theorem:dvrDetails}, we need to show
that the base-change of the top row in the
diagram~\eqref{equation:pullbackOfExt} gives a generator of 
$\homology^n(X_S, \omega_{X_S})$ for each ring map $R \to S$. 
(Here $X_S := X \times_{\Spec R} \Spec S$.) 
We need to show two things:
\begin{asparaenum}
\item
$\left(0 \to  \omega_X \to \cdots \to  
\widetilde{K}_2 \to 
\calG_1 \to 
\strSh_{X} \to  0\right) \otimes_R S$ is exact.
For this, we need to show that the sheaves are flat over $R$.
The terms other than $\calG_1$ are locally free over $X$, so it remains to
show that $\calG_1$ is torsion-free over $R$. 
First note that $\calG_0 = I_{Z'} \otimes \tilde K_1$ is torsion-free.
Let $g$ be a (local) section of $\calG_1$ that is torsion over $R$.
The images of $g$ in $\strSh_X$ and in $\calG_0$ are zero since both these
sheaves are torsion-free over $R$. Since $\calG_1$ is the pull-back, it
follows that $g =0$.

\item
$\left(0 \to  \omega_X \to \cdots \to  
\widetilde{K}_2 \to 
\calG_1 \to 
\strSh_{X} \to  0\right) \otimes_R S$ is non-zero in 
$\Ext_{X_S}^n(\strSh_{X_S}, \omega_{X_S})$.
Suppose it is zero for some $S$. Then it would be zero for some field $F$
with a ring map $S \to F$, so we may assume that $S$ is a field.
In this case, we immediately reduce (since maps between fields are
faithfully flat) to the case that $S$ is either the residue field of $R$ or
the fraction field of $R$.
Since the fraction field of $R$ is flat over $R$ and 
$\Ext_{X}^n(\strSh_{X}, \omega_{X}) \simeq R$, the assertion follows in
this case.
Hence assume that $S = \Bbbk$.

We first argue that if we apply 
$- \otimes_R \Bbbk$, to the 
diagram~\eqref{equation:pullbackOfExt} for $R$, we get
diagram~\eqref{equation:pullbackOfExt} for $\Bbbk$.
Indeed, since the rows consist of sheaves that are torsion-free over $R$
($Z = \Spec R$ and $X_n$ is flat over $R$),
the rows remain exact after applying $- \otimes_R \Bbbk$.
It is also immediate that $\strSh_{X}, \strSh_{Z}, \strSh_{X_n}$ are
replaced by $\strSh_{X_s}, \strSh_{Z_s},
\strSh_{(X_n)_s}$, respectively.
Hence it remains to show that $\omega_X \otimes_R \Bbbk = 
\omega_{X_s}$.
This is indeed true,
since $X_s$ is the pull-back of the divisor $\Spec \Bbbk
\subseteq \Spec R$, and, therefore, 
the normal bundle of $X_s$ inside $X$ is isomorphic to 
$\strSh_{X_s}$.
Therefore the base-change of the 
diagram~\eqref{equation:pullbackOfExt} for $R$ along the map $R \to \Bbbk$
gives the diagram~\eqref{equation:pullbackOfExt} for $\Bbbk$.
By Theorem~\ref{theorem:fld}\eqref{theorem:fld:extgen}, the top row of 
the diagram~\eqref{equation:pullbackOfExt} for $\Bbbk$ generates 
$\Ext_{X_s}^n(\strSh_{X_s},
\omega_{X_s})$ as an $\Bbbk$-vector space.
\end{asparaenum}

This concludes the proof of
Theorem~\protect{\ref{theorem:dvrDetails}}.
\end{proof}

\section{Grassmannian}
\label{section:grass}

Let $m \geq 4$ and 
$X$ the Grassmannian $G_{2,m}$ over $\ints$, i.e., the scheme whose
$R$-points (for a ring $R$) classify the locally free (over $R$) quotient
modules of rank $m-2$ of $R^{\oplus m}$. It is a smooth $\ints$-scheme of
relative dimension $n := 2(m-2)$.
We show that the morphism $X \to \Spec \ints$ has a section $Z$ and that
there exists a complete intersection $X_n$ inside $X$ of relative dimension
$0$ over $\ints$ satisfying the conditions of Notation~\ref{notation:main}.
Since $\ints$ is a PID, the strategy of Section~\ref{section:strategy} can
be carried out to give a generator of 
$\Ext_X^n(\strSh_X, \omega_X)$ as a free $\ints$-module.

Let $N = \binom m2 -1$.
Then $X$ can be embedded inside $\projective := 
\projective_\ints^N$ as a closed
subscheme, defined by the Pluecker relations;
see~\cite[Chapter~III, \S 2.7]{EisenbudHarrisSchemes2000}.
Let $p_{i,j}, 1 \leq i < j \leq m$ be homogeneous coordinates for 
$\projective$.
Let $S = \ints[p_{i,j}, 1 \leq i < j \leq m]$.
The main result of this section is:

\begin{theorem}
\label{theorem:grassDetails}
For $3 \leq k \leq 2m-1$, write $l_k = \sum_{i+j=k} p_{i,j}$.
Write $V \subseteq \projective$ for the linear subvariety defined by $l_3,
l_4, \ldots, \widehat{l_{m+1}}, \ldots, l_{2m-1}$.
Then $X \cap V$ has zero-dimensional fibres and there is a section 
$\Spec \ints \to X \cap V$ which defines an irreducible and 
reduced component of $X \cap V$.
\end{theorem}

To prove that $X \cap V$ has zero-dimensional fibres, we may do a
base-change along $\ints \to \Bbbk$ where $\Bbbk$ is an algebraically
closed field, and show that $X_\Bbbk \cap V_\Bbbk$ 
(the respective base-changes)
is a zero-dimensional scheme. We do this now.
We use~\cite[Chapter~1]{SeshadriSMT07} as the reference on Schubert
varieties and standard monomial theory.
We denote the Schubert varieties in 
$X_\Bbbk$ by $Y_{i,j}$ $1 \leq i < j \leq n$.
(This is the Schubert variety in $X_\Bbbk$ corresponding to a permutation
$\sigma \in S_n$ with $\{\sigma(1), \sigma(2)\} = \{1,2 \}$.)

\begin{proposition}
\label{proposition:G2nFld}
Let $\Bbbk$ be an algebraically closed field. 
Then $X_\Bbbk \cap V_\Bbbk$ is a zero-dimensional scheme.
\end{proposition}

\begin{proof}
We think of $X_\Bbbk$ as the quotient of the space of rank-two $m \times 2$
matrices under the natural action of $\GL_2(\Bbbk)$ on the right.
The homogeneous coordinate $p_{i,j}$ of $\projective_\Bbbk$ becomes, on
$X_\Bbbk$, the determinant of the $2 \times 2$ minor consisting of the
rows $i$ and $j$ and both the columns.

Note that $V_\Bbbk$ is defined inside $\projective_\Bbbk$ by (the images in
$S \otimes_\ints \Bbbk$) of the $l_k$.
For each $m+2 \leq k \leq 2m-1$, let $W_k$ be the linear 
subvariety defined by $l_k, \ldots, l_{2m-1}$.
We claim that 
$X_\Bbbk \cap W_k$ is set-theoretically the union of all the 
Schubert subvarieties of $X_\Bbbk$ of dimension $k-4$.
Note that $k-4 = \dim X_\Bbbk  - \codim_{\projective_\Bbbk} W_k$.
Therefore the irreducible components of $X_\Bbbk \cap W_{m+2}$ are
precisely the Schubert varieties $Y_{k,m+1-k}$ for all $1 \leq k < 
\frac m 2$.

For each $3 \leq k \leq m$, let $W'_k$ be the 
linear subvariety defined by $l_3, \ldots,l_{k}$.
By symmetry, and assuming the above claim, 
$X_\Bbbk \cap W'_k$ is set-theoretically the union of all the 
opposite Schubert subvarieties of $X_\Bbbk$ of codimension $k-2$.
In particular, the irreducible components of $X_\Bbbk \cap W'_{m}$ are
precisely the opposite Schubert varieties $Y^{k,m+1-k}$ for all $1 \leq k < 
\frac m 2$.
Hence the irreducible components of 
$X_\Bbbk \cap V_\Bbbk = (X_\Bbbk \cap W_{m+2}) \cap (X_\Bbbk \cap W'_{m})$ 
are the Richardson varieties $Y^{k,m+1-k}_{k,m+1-k}$ for all $1 \leq k <
\frac m 2$, all of which are zero-dimensional. (See,
e.g.,~\cite[\S~1.3]{BrionGeomFlagVars2005}.)

To prove the claim, we proceed by downward induction on $k$. When $k=2m-1$,
$X_\Bbbk \cap W_k$ is the Schubert divisor $Y_{m-2,m}$, which is of
dimension $2m-5$. Assume that the claim has been proved up to $k+1$.
Let $Y$ be a Schubert variety of dimension $k-3$. 
Hence $Y$ is an irreducible component of $X_\Bbbk \cap W_{k+1}$.
Say $Y = Y_{i,k-i}$. Then, in the terms of $l_k$, all but $l_{i,k-i}$
vanish along $Y$, since the ideal of $Y_{i,k-i}$ in $X_\Bbbk$ is generated
by $\{p_{j_1,j_2} \mid (j_1, j_2 ) \not \leq (i, k-i)\}$. Moreover, 
$Y_{i,k-i} \cap \{p_{i,k-i } = 0\}$ is the union of the codimension-$1$
Schubert subvarieties of $Y_{i,k-i}$. Further, every $(k-4)$-dimensional
Schubert variety is a subvariety of a $(k-3)$-dimensional Schubert variety.
Hence 
$X_\Bbbk \cap W_k = X_\Bbbk \cap W_{k+1} \cap \{l_k = 0 \}$
is set-theoretically the union of all the 
$(k-4)$-dimensional Schubert varieties.
\end{proof}

\begin{proof}[Proof of Theorem~\protect{\ref{theorem:grassDetails}}]
As we remarked above,
Proposition~\ref{proposition:G2nFld}
implies that $X \cap V$ has zero-dimensional fibres.
We now prove the existence of a section of the structure morphism $X \cap
V \to \Spec \ints$ with the desired properties.
Let $\affine \subseteq \projective$ be the affine space obtained by
inverting $p_{1,n}$. More precisely, 
write $q_{i,j} = \frac{p_{i,j } }{p_{1,n } }$,
$R = \ints[q_{i,j} : 1 \leq i<j\leq n, (i,j) \neq (1,n)]$
and $\frakm$ for the prime ideal of $R$ generated by the $q_{i,j}$.
Then $\affine = \Spec R$.
Note that $X$ is defined by the 
the Pluecker relations, which are of the form
\[
p_{i,j }p_{i',j' } \pm p_{i,i'} p_{j',j} \pm p_{i,j' }p_{i',j}
\;\text{where}\; i<i', j>j'.
\]
See, e.g.,~\cite[Chapter~III, \S 2.7]{EisenbudHarrisSchemes2000}.
(For the sake of convenience, we do not specify
which terms on the right appear with positive sign and which with negative
sign.)
A Pluecker relation of the form
\[
p_{1,n }p_{i',j' } \pm p_{1,k} p_{j',n} \pm p_{1,j' }p_{k,n}, \;\text{with}\;  
1 < i' < j' < n
\]
gives a generator 
\[
q_{i',j' } \pm q_{1,k} q_{j',n} \pm q_{1,j' }q_{k,n}
\]
of the ideal of $X \cap \affine$ inside $\affine$.
The other generators (i.e., with $i<i', j>j', (i,j) \neq (1,n)$) give
the polynomials 
$q_{i,j }q_{i',j' } \pm q_{i,k} q_{j',j} \pm q_{i,j' }q_{k,j}$
inside the ideal of $X \cap \affine$.
For $3 \leq k \leq 2m-1, k \neq n+1$, the polynomial $l_k$ gives the
polynomial $q_{1,k-1} + q_{2,k-2} + \cdots \in R$.

We need to show that $\frakm$ is a minimal prime over 
the $R$-ideal $\fraka$ generated by
\begin{enumerate}

\item $q_{i',j' } \pm q_{1,i'} q_{j',n} \pm q_{1,j' }q_{i',n}$, 
$1 < i' < j' < n$;

\item
$q_{i,j }q_{i',j' } \pm q_{i,i'} q_{j',j} \pm q_{i,j' }q_{i',j}$,
$i<i', j>j', (i,j) \neq (1,n)$;

\item
\begin{enumerate}

\item
$q_{1,k-1} + q_{2,k-2} + \cdots$, 
$3 \leq k \leq n$;

\item
$\cdots + q_{k-n+1,n-1} + q_{k-n,n}$, $n+2 \leq k \leq 2m-1$ 

\end{enumerate}
\end{enumerate}
and that $\fraka R_\frakm = \frakm R_\frakm$.
Note that $\fraka \subseteq \frakm$.
Therefore if we show that $\Omega_{(S/\fraka)/\ints}$ vanishes at $\frakm$,
it would follow that $S/\fraka$ is unramified over $\ints$ at $\frakm$,
whence it would follow that $\fraka R_\frakm = \frakm R_\frakm$.
Note that $\Omega_{(S/\fraka)/\ints}$ is the cokernel of the jacobian matrix 
of the aforementioned generating set of $\fraka$ with respect to the
variables $q_{i',j'}$.
Further, for each $(i',j' )$, there exists a unique generator in
the above list in which $q_{i',j'}$ appears with coefficient $1$, so the
jacobian matrix has full rank in the residue field $\kappa(\frakm)$ at
$\frakm$.
Therefore $\Omega_{(S/\fraka)/\ints}$ vanishes at $\frakm$.
\end{proof}

\begin{proof}[Proof of Theorem~\protect{\ref{theorem:grass}}]
Theorem~\ref{theorem:grassDetails} implies that the conditions on 
Notation~\ref{notation:main} are met.
Hence we can apply Proposition~\ref{proposition:toprowPID} to complete the
proof of Theorem~\ref{theorem:grass}.
\end{proof}

\begin{remarkbox}
\label{remarkbox:Gdn}
One could ask whether the above approach will work for more general
Grassmannians. Let $\Bbbk$ be a field. Consider the Grassmannian $G_{d,n }$
of $d$-dimensional subspaces of $\Bbbk^n$.
Let $I_{d,m} = \{1 \leq i_1 < \cdots < i_d \leq m \}$.
Let $p_\bfi, \bfi \in I_{d,m}$ be Pluecker coordinates.
For integers $s$ with $\binom{d+1}2 \leq s \leq \binom{d+1}2 + d(m-d)$,
let
\[
l_s := \sum_{|\bfi|=s } p_\bfi
\]
Then $l_s, \binom{d+1}2 \leq s \leq \binom{d+1}2 + d(m-d)$ form a regular
sequence on the coordinate ring $R$ of $G_{d,m }$. There are $d(m-d)+1$
linear forms. To imitate the above proof, we need to show that a subset
consisting of $d(m-d)$ linear forms defines 
a zero-dimensional subscheme of $G_{d,m}$ that has as one component 
a reduced $\Bbbk$-rational point. Computations in~\cite{M2} for $G_{3,6}$ 
showed that for every choice of $9$ linear forms (note that $\dim
G_{3,6}=9$), none of the components of the corresponding
zero-dimensional subscheme is reduced.
\end{remarkbox}

\section{An example}
\label{section:EulerKoszul}

In the Introduction, we mentioned that, in general, for a smooth connected
projective scheme $X$ over $R$ with a closed embedding $i : X \to
\projective^N_R$, the Koszul complex constructed from the vector-bundle on
$X$ (and a section) induced by the pull-back of the Euler sequence on
$\projective^N_R$ need not give a generator of 
$\Ext_X^n(\strSh_X, \omega_{X/R})$. We now describe an example of this
behaviour.

Let $R = \ints$ and $X =G_{2,4}$.
Write $\projective^5 = \projective^5_\ints$.
Let $H_1, \ldots, H_4$ be the hyperplanes in 
$\projective^5$ defined by the linear polynomials $l_3, l_4, l_6, l_7$
(defined in the previous section) respectively.
Let $D_i = X \cap H_i$, $1 \leq i \leq 4$.

\begin{proposition}
\label{proposition:isoCohOmega}
$\homology^i(\projective^5, \Omega_{\projective^5}(m)) 
\simeq \homology^i(X, \Omega_{X}(m))$ for all $i$ and for $m=0,1$.
\end{proposition}

\begin{proof}
Let $I$ be the ideal defining $X$ inside $\projective^5$.
It is generated by a degree $2$ polynomial.
Hence $I/I^2 \simeq\strSh_X(-2)$.
Take the conormal sequence twisted by $\strSh_X(m)$:
\[
0 \to \strSh_X(m-2) \to \Omega_{\projective^5}|_{X}(m) \to 
\Omega_{X}(m) \to 0.
\]
We see that $\homology^*(X,\strSh_X(m-2)) = 0$ from the exact sequence
\[
0 \to \strSh_{\projective^5}(m-4)
\to \strSh_{\projective^5}(m-2)
\to \strSh_X(m-2) \to 0.
\]
Hence
$\homology^*(X,\Omega_{\projective^5}|_{X}(m))
=\homology^*(X,\Omega_{X}(m))$.

To calculate $\homology^*(X,\Omega_{\projective^5}|_{X}(m))$,
consider the exact sequence (obtained by applying 
$-\otimes \Omega_{\projective^5}(2)$ to the above sequence)
\[
0 \to \Omega_{\projective^5}(m-2)
\to \Omega_{\projective^5}(m)
\to \Omega_{\projective^5}|_{X}(m) \to 0.
\]
We want to show that 
$\homology^*(\projective^5,\Omega_{\projective^5}(m-2)) =
0$, which we get
from the Euler exact sequence (after a twist)
\[
0 \to \Omega_{\projective^5}(m-2) \to 
\strSh_{\projective^5}(m-3)^6 \to \strSh_{\projective^5}(m-2)
\to 0.
\]
\end{proof}

\begin{corollary}
$\homology^1(X, \Omega_{X}) = \ints$.
\end{corollary}

\begin{proof}
Follows from Proposition~\ref{proposition:isoCohOmega}
and the Euler exact sequence.
\end{proof}

\begin{proposition}
Let $1 \leq i \leq 4$.
The natural map $\Ext^1(\strSh_{D_i},\omega_X) \to
\Ext^1(\strSh_X,\omega_X) = \homology^1(X,\strSh_X)$
is an isomorphism. 
\end{proposition}

\begin{proof}
Write $D = D_i$.
From the Euler exact sequence, we get that
$\homology^0(\projective^5,\Omega_{\projective^5}(1)) = 
\homology^1(\projective^5,\Omega_{\projective^5}(1)) = 0$.
Hence Proposition~\ref{proposition:isoCohOmega} gives that
$\homology^0(X,\Omega_X(1)) = 
\homology^1(X,\Omega_X(1)) = 0$.
From the exact sequence
\[
0 \to \strSh_X(-1) \to \strSh_X \to \strSh_D \to 0
\]
we get the exact sequence
\[
\Hom_X (\strSh_X(-1), \Omega_X) \to 
\Ext^1(\strSh_D,\omega_X) \to \Ext^1(\strSh_X,\omega_X)
\to \Ext^1(\strSh_X(-1),\omega_X).
\]
Since the first and the last terms are zero, we get the isomorphism.
\end{proof}

Fix a generator
\[
0 \to \Omega_X \to \calF \stackrel{s}\to  \strSh_X \to 0
\]
of $\Ext^1(\strSh_X,\omega_X)$.
(Note that $s$ is a global section of $\calF^*$.)
Let 
\begin{equation}
\label{equation:defFD}
0 \to \Omega_X \to \calF_{D_i} \stackrel{s_i} \to \strSh_{D_i} \to 0
\end{equation}
be its inverse image in $\Ext^1(\strSh_{D_i},\omega_X)$.

In what follows, for a chain complex $L_\bullet$, 
$\bfh_j(L_\bullet)$ denotes the $j$th homology of $L_\bullet$, i.e., 
\[
\frac{\ker (L_j \to L_{j-1})} {\image (L_{j+1} \to L_{j})}.
\]

\begin{lemma}
\label{lemma:tensorFD}
Let $1 \leq i \leq 4$.
Let $L_\bullet$ be a bounded complex such that 
$\Tor_k(L_j, \strSh_{D_i}) = 0$ for all $j$ and $k>0$.
Write $L'_\bullet$ for the 
total complex of $L_\bullet \otimes (\calF_{D_i} \to \strSh_{D_i})$.
Then 
$\bfh_j(L'_\bullet) = \bfh_{j-1}(L_\bullet) \otimes \Omega_X$.
\end{lemma}

\begin{proof}
$L_\bullet \otimes (\calF_{D_j} \to \strSh_{D_j})$
gives a third-quadrant double complex 
\[
E^{-p,-q} 
\begin{cases}
L_p \otimes \strSh_{D_j}, & q = 0
\\
L_p \otimes \calF_{D_j}, & q = 1
\\
0, & \text{otherwise}.
\end{cases}
\]
Since the vertical map $E^{-p,-q} \to E^{-p,-q+1}$ are surjective 
and $\Tor_1(L_p, \strSh_{D_j}) = 0$, the
${}_vE_1$ page of the spectral sequence is
\[
{}_vE_1^{-p,-q}
\begin{cases}
L_p \otimes \Omega_X, & q = 1
\\
0, & \text{otherwise}.
\end{cases}
\]
Hence $L'_\bullet$ is quasi-isomorphic to $L_\bullet \otimes \Omega_X[-1]$.
\end{proof}

\begin{lemma}
\label{lemma:torAssoc}
Let $\calF_1, \calF_2, \calF_3$ be coherent sheaves on $X$ such that 
$\Tor_{\geq 1}^{\strSh_X} (\calF_1 , \calF_2) = 0 = \Tor_{\geq
1}^{\strSh_X} (\calF_2 , \calF_3)$.
Then for all $i$, 
then $\Tor_i^{\strSh_X} (\calF_1 \otimes \calF_2, \calF_3) = 
\Tor_i^{\strSh_X} (\calF_1 , \calF_2 \otimes \calF_3)$ for all $i$.
\end{lemma}

\begin{proof}
Let $\calL_1$ and $\calL_3$ be locally free resolutions of $\calF_1$ and
$\calF_3$ respectively.
Then we have an isomorphism
$(\calL_1 \otimes \calF_2) \otimes \calL_3 \to
\calL_1 \otimes (\calF_2 \otimes \calL_3)$.
By hypothesis, 
$\calL_1 \otimes \calF_2$ is quasi-isomorphic to 
$\calF_1 \otimes \calF_2$, so 
$\bfh_*((\calL_1 \otimes \calF_2) \otimes \calL_3) = 
\Tor_*^{\strSh_X} (\calF_1 \otimes \calF_2, \calF_3)$.
Similarly,
$\bfh_*(\calL_1 \otimes (\calF_2 \otimes \calL_3)) = 
\Tor_*^{\strSh_X} (\calF_1, \calF_2 \otimes \calF_3)$.
\end{proof}

\begin{lemma}
\label{lemma:higherTorMulti}
Let $A \subsetneq \{1, \ldots, 4 \}$ and $b \in \{1, \ldots, 4 \} \minus
A$. Then
\[
\Tor_{\geq 1}^{\strSh_X} \left(\bigotimes_{a \in A} \calG_a,
\calG_b\right) = 0.
\]
for all $\displaystyle (\calG_1, \ldots, \calG_4) \in \prod_{i=1}^4 
\{\strSh_{D_i}, \calF_{D_i}\}$.
\end{lemma}

\begin{proof}
We may assume that $\calG_b = \strSh_{D_b}$,
since for every quasi-coherent
sheaf $\calG$, $\Tor_{\geq 1}^{\strSh_X} (\calG, \calF_{D_b})
\subseteq
\Tor_{\geq 1}^{\strSh_X} (\calG, \strSh_{D_b})$ by~\eqref{equation:defFD}.
We proceed by induction on $(|A|, \delta_A)$, where 
$\delta_A = \left|\{ a \in A \mid \calG_a = \calF_{D_a} \}\right|$. 
We order such pairs lexicographically.

Suppose that $\delta_A = 0$. 
Since the coordinate ring of $X$ is Cohen-Macaulay, the homogeneous
polynomials $l_3, l_4, l_6, l_7$ form a regular sequence in any order.
Therefore, irrespective of $|A|$, the assertion of the lemma follows.
When $(|A|, \delta_A) = (1,1)$, 
we use the fact that by~\eqref{equation:defFD},
\[
\Tor_{\geq 1}^{\strSh_X} (\calF_{D_a},\strSh_{D_b} )
\subseteq
\Tor_{\geq 1}^{\strSh_X} (\strSh_{D_a}, \strSh_{D_b})
\]
which is zero by the case $\delta_A = 0$.

Now assume that $|A| > 1$ and $\delta_A > 0$.
Let $a_1 \in A$ be such that $\calG_a = \calF_{D_a}$ for some 
$a \in A' := A \minus \{a_1 \}$.
By induction on $|A|$,
\[
\Tor_{\geq 1}^{\strSh_X} \left(\bigotimes_{a \in A'} \calG_a,
\calG_{a_1}\right) = 0
\quad\text{and}\quad
\Tor_{\geq 1}^{\strSh_X} \left(\calG_{a_1}, \strSh_{D_b}\right) = 0
\]
Hence by Lemma~\ref{lemma:torAssoc}, 
\[
\Tor_i^{\strSh_X} \left(\bigotimes_{a \in A}\calG_a, \strSh_{D_b} \right) = 
\Tor_i^{\strSh_X} \left(\bigotimes_{a \in A'}\calG_a, 
\calG_{a_1} \otimes \strSh_{D_b}\right)
\]
for all $i$.
Repeating this, we find $\tilde a \in A$ such that $\calG_{\tilde a} =
\calF_{D_{\tilde a}}$ and such that
\[
\Tor_i^{\strSh_X} \left(\bigotimes_{a \in A}\calG_a, \strSh_{D_b} \right) = 
\Tor_i^{\strSh_X} \left(\calF_{D_{\tilde a}}, 
\bigotimes_{\substack{a \in A\\ a \neq \tilde a}}
\calG_a \otimes \strSh_{D_b} \right) = 0
\]
for all $i \geq 1$, since 
$\delta_{(A \setminus \{\tilde a\}) \cup \{b\}} = \delta_A-1$.
\end{proof}

\begin{proposition}
\label{proposition:exactSeqalphaDivs}
$\Tot \left ( \bigotimes_{k=1}^4 (\calF_{D_k} \stackrel{s_i}\to 
\strSh_{D_k})\right)$ 
gives a six-term exact sequence of the form
\[
\alpha : \quad
0 \to \Omega_X^{\otimes 4} \to \cdots \to 
\bigoplus_{i=0}^4 
\mathcal{F}_i|_{D_1\cap \cdots  \widehat{D_i} \cdots \cap D_4}
\to
\strSh_{D_1 \cap \cdots \cap D_4} \to 0
\]
\end{proposition}

\begin{proof}
We prove this by repeated application of Lemma~\ref{lemma:tensorFD}.
Let $L^{(1)}_\bullet = 0 \to \calF_{D_1} \to \strSh_{D_1} \to 0$
and, define, for $2 \leq i \leq 4$,
Let $L^{(i)}_\bullet =  \Tot
(L^{(i-1)}_\bullet \otimes (\calF_{D_i} \to \strSh_{D_i}))$.
Note that for each $1 \leq i \leq 3$, 
$L^{(i)}_\bullet$ satisfies the hypotheses of 
Lemma~\ref{lemma:tensorFD}: Firstly, for each $j$, 
$L^{(i)}_j$ is a tensor product of a few $\calF_{D_k}$ and a few
$\strSh_{D_k}$, so we can use Lemma~\ref{lemma:higherTorMulti}.
Secondly
by induction on $i$, we see that $\bfh_j(L^{(i)}_\bullet) = 0$ unless
$j=i$, and $\bfh_i(L^{(i)}_\bullet) \simeq \Omega_X^{\otimes i}$.
\end{proof}

\begin{proposition}
\label{proposition:exactSeqbetaX}
$\Tot \left ( \bigotimes_{k=1}^4 (\calF \stackrel s \to \strSh_X)\right)$ 
gives a six-term exact sequence of the form
\[
\beta : \quad
0 \to \Omega_X^{\otimes 4} \to \cdots \to 
\bigoplus_{i=0}^4 \mathcal{F} \to
\strSh_X \to 0.
\]
\end{proposition}

\begin{proof}
Let $L_\bullet$ be a bounded complex.
Write $L'_\bullet 
= \Tot(L_\bullet \otimes (\calF \stackrel s\to \strSh_X))$.
Then 
$\bfh_j(L'_\bullet) = \bfh_{j-1}(L_\bullet) \otimes \Omega_X$.
Now repeatedly apply this to 
$\Tot \left ( \bigotimes_{k=1}^r (\calF \stackrel s \to \strSh_X)\right)$,
$2 \leq r \leq 4$.
\end{proof}

We now obtain the following diagram of complexes, where the 
first row is $\alpha$ from
Proposition~\ref{proposition:exactSeqalphaDivs}
the second row is $\beta$ from Proposition~\ref{proposition:exactSeqbetaX}
and the last row $\gamma$ is the \emph{exact} 
Koszul complex for the map $\calF \stackrel s \to \strSh_X$.
\[
\xymatrix@C=4ex{
\alpha : & 
0 \ar[r] & \Omega_X^{\otimes 4} \ar[r] & \bigotimes_{i=0}^4 \mathcal{F}_i
\ar[r]  & \cdots \ar[r] & \bigoplus_{i=0}^4 \mathcal{F}_i|_{D_1\cap \cdots
\widehat{D_i} \cdots \cap D_4} \ar[r]  &  \strSh_{D_1 \cap \cdots \cap
D_4}\ar[r] & 0
\\
\beta : & 
0 \ar[r] & \Omega_X^{\otimes 4} \ar[r]\ar[d]^\mu \ar@{=}[u] & \bigotimes^4
\mathcal{F} \ar[r]\ar[d]\ar[u]  & \cdots \ar[r] & \bigoplus^4 \mathcal{F}
\ar[r]\ar[d]\ar[u]  &  \strSh_X\ar[r]\ar@{=}[d]\ar[u]^\pi & 0 
\\
\gamma : & 
0 \ar[r] & \omega_X \ar[r] & \bigwedge^4 \mathcal{F} \ar[r]  & \cdots
\ar[r] & \bigwedge^1 \mathcal{F} \ar[r]  &  \strSh_X\ar[r] & 0 
}
\]
($\pi$ and $\mu$ are the natural maps.)
Note that 
$\alpha \in\Ext_X^4(\strSh_{D_1 \cap \cdots \cap D_4}, 
\Omega_X^{\otimes 4})$,
$\beta \in \Ext_X^4(\strSh_X,\Omega_X^{\otimes 4})$ and 
$\gamma \in \Ext_X^4(\strSh_X,\omega_X)$. 

\begin{lemma}
\label{lemma:pialaphamubeta}
$\pi^*(\alpha) = \beta$ and $\mu_*(\beta)  = \gamma$.
\end{lemma}

\begin{proof}
Use~\cite[Proposition~III.5.1]{MacLHomology63} to see that 
$\pi^*(\alpha) = (\mathrm{id}_{\Omega_X^{\otimes 4}})_*(\beta)=  \beta$
and that
$\mu_*(\beta)  = (\mathrm{id}_{\strSh_X})^*(\gamma) = \gamma$.
\end{proof}

\begin{lemma}
$\pi^*(\mu_*(\alpha)) = \mu_*(\pi^*(\alpha)) = \gamma$.
\end{lemma}

\begin{proof}
The second equality is from Lemma~\ref{lemma:pialaphamubeta}.
To prove the first equality, let $\calA$ be an abelian category with enough
injectives. Write $\mathrm{Ch}(\calA)$ for the abelian category of chain
complexes over $\calA$.
Let $A,B, C, D$ be objects in $\calA$ and $\pi : C \to A$ and $\mu : B \to
D$ be morphisms. Let $n \in \naturals$.
We want to show that $\pi^*\mu_* = \mu_* \pi^*$ 
as maps from $\Ext_\calA^n(A,B)$ to $\Ext_\calA^n(C,D)$.
Let $I_B$ and $I_D$ be injective resolutions of $B$ and $D$ respectively.
Then the following diagram commutes
\[
\xymatrix{
\Hom_{\mathrm{Ch}(\calA)}(A,I_B) \ar[r]^{\pi^*} \ar[d]_{\mu_*} &
\Hom_{\mathrm{Ch}(\calA)}(C,I_B) \ar[d]_{\mu_*}
\\
\Hom_{\mathrm{Ch}(\calA)}(A,I_D) \ar[r]^{\pi^*}  &
\Hom_{\mathrm{Ch}(\calA)}(C,I_D)
}
\]
Taking homology we see that $\pi^*\mu_* = \mu_* \pi^*$ 
as maps from $\Ext_\calA^n(A,B)$ to $\Ext_\calA^n(C,D)$.
\end{proof}

\begin{proposition}
\label{proposition:TwoPts}
With notation as above:
\begin{enumerate}

\item
\label{proposition:TwoPts:red}
$D_1 \cap \cdots \cap D_4$ is a reduced scheme $\{P_1,P_2 \}$ where, for
$i=1,2$, $P_i \simeq \Spec \ints$.

\item
\label{proposition:TwoPts:aut}
The automorphism $\tau$ of $X$ given by $p_{i,j } \mapsto p_{k,l }$ where
$\{i,j,k,l \} = \{1, \ldots, 4 \}$ sends $P_1$ to $P_2$ and vice versa.
\end{enumerate}

\end{proposition}

\begin{proof}
\eqref{proposition:TwoPts:red}:
In the Pluecker embedding, $X = G_{2,4 }$ is defined inside
$\projective^5 := \Proj \ints[p_{1,2}, \ldots, p_{3,4 }]$
by the polynomial $p_{1,2 }p_{3,4 } - p_{1,3 }p_{2,4 } + p_{1,4 }p_{2,3}$.
Hence 
$D_1 \cap \cdots \cap D_4$
is defined by 
$(p_{1,2 }p_{3,4 } - p_{1,3 }p_{2,4 } + p_{1,4 }p_{2,3}, p_{1,2 },
p_{1,3}, p_{2,4}, p_{3,4}) 
=(p_{1,2 }, p_{1,3}, p_{2,4}, p_{3,4}, p_{1,4 }p_{2,3})
$.

\eqref{proposition:TwoPts:aut}:
Note that 
\[
(p_{1,2 }, p_{1,3}, p_{2,4}, p_{3,4}, p_{1,4 }) \mapsto
(p_{1,2 }, p_{1,3}, p_{2,4}, p_{3,4}, p_{2,3})
\]
 and vice versa.
\end{proof}

\begin{proposition}
$\gamma$ is an even number when we identify $\Ext_X^4(\strSh_X, \omega_X)$
with $\ints$. In particular it is not a generator of 
$\Ext_X^4(\strSh_X, \omega_X)$.
\end{proposition}

\begin{proof}
Note that 
$\Ext_X^4(\strSh_{D_1 \cap \cdots \cap D_4}, \omega_X)
\simeq 
\Ext_X^4(\strSh_{P_1} , \omega_X) \oplus 
\Ext_X^4(\strSh_{P_2}, \omega_X) \simeq 
\ints\varepsilon_1 \oplus \ints \varepsilon_2$.
(The $\varepsilon_i$ are basis elements.)
The automorphism $\tau$ of $X$ defined in
Proposition~\ref{proposition:TwoPts} gives an isomorphism 
$\tau : \ints\varepsilon_1 \stackrel{\simeq }\to \ints \varepsilon_2$,
with
$\varepsilon_1 \mapsto \varepsilon_2$ and
$\varepsilon_2 \mapsto \varepsilon_1$.
Write $\mu_*(\alpha) = m_1\varepsilon_1 + m_2 \varepsilon_2$.
Hence $\tau^*\mu_*(\alpha)  = m_2\varepsilon_1 + m_1 \varepsilon_2$.
We see from Proposition~\ref{proposition:TwoPts} that
$\mu_*(\tau^*\alpha) = m_2\varepsilon_1 + m_1 \varepsilon_2$.
On the other hand, note that $\tau$ permutes the divisors $D_1, \ldots,
D_4$. Hence, we see from the definition of $\alpha$ 
(Proposition~\ref{proposition:exactSeqalphaDivs}) that
$\tau^* \alpha  = \alpha$.
Therefore $m_1 = m_2$, so
$\mu_* \alpha = m_1 (\varepsilon_1 + \varepsilon_2)$.
Thus we see that it is enough to show that  
$\pi^*(\varepsilon_1 + \varepsilon_2)$ is not a generator of 
$\Ext_X^4(\strSh_X, \omega_X)$.

Write $\pi_i^*$ for the isomorphism
$\Ext_X^4(\strSh_{P_i} , \omega_X) \to 
\Ext_X^4(\strSh_X, \omega_X)$, $i=1,2$.
Then $\pi^*(\varepsilon_1 + \varepsilon_2) 
= \sum_i \pi_i^*(\varepsilon_i)$. 
Since $\pi_i$ is an isomorphism, 
$\pi_i^*(\varepsilon_i)$ 
is a generator of $\Ext_X^4(\strSh_X, \omega_X) \simeq \ints$.
Therefore $\pi^*(\varepsilon_1 + \varepsilon_2)$ is an even 
multiple of a generator of $\Ext_X^4(\strSh_X, \omega_X)$.
\end{proof}

\def\cfudot#1{\ifmmode\setbox7\hbox{$\accent"5E#1$}\else
  \setbox7\hbox{\accent"5E#1}\penalty 10000\relax\fi\raise 1\ht7
  \hbox{\raise.1ex\hbox to 1\wd7{\hss.\hss}}\penalty 10000 \hskip-1\wd7\penalty
  10000\box7}


\begin{thebibliography}{Mat89}

\bibitem[Bri05]{BrionGeomFlagVars2005}
M.~Brion.
\newblock Lectures on the geometry of flag varieties.
\newblock In {\em Topics in cohomological studies of algebraic varieties},
  Trends Math., pages 33--85. Birkh\"auser, Basel, 2005.
\newblock arXiv:math/0410240 [math.AG].

\bibitem[CE99]{CartanEilenberg1999}
H.~Cartan and S.~Eilenberg.
\newblock {\em Homological algebra}.
\newblock Princeton Landmarks in Mathematics. Princeton University Press,
  Princeton, NJ, 1999.
\newblock With an appendix by David A. Buchsbaum, Reprint of the 1956 original.

\bibitem[EH00]{EisenbudHarrisSchemes2000}
D.~Eisenbud and J.~Harris.
\newblock {\em The geometry of schemes}, volume 197 of {\em Graduate Texts in
  Mathematics}.
\newblock Springer-Verlag, New York, 2000.

\bibitem[Eis95]{eiscommalg}
D.~Eisenbud.
\newblock {\em Commutative algebra, with a View Toward Algebraic Geometry},
  volume 150 of {\em Graduate Texts in Mathematics}.
\newblock Springer-Verlag, New York, 1995.

\bibitem[Har66]{HartRD66}
R.~Hartshorne.
\newblock {\em Residues and duality}.
\newblock Lecture notes of a seminar on the work of A. Grothendieck, given at
  Harvard 1963/64. With an appendix by P. Deligne. Lecture Notes in
  Mathematics, No. 20. Springer-Verlag, Berlin, 1966.

\bibitem[Har77]{HartAG}
R.~Hartshorne.
\newblock {\em Algebraic geometry}.
\newblock Springer-Verlag, New York, 1977.
\newblock Graduate Texts in Mathematics, No. 52.

\bibitem[Lip84]{LipmanAsterisque1984}
J.~Lipman.
\newblock Dualizing sheaves, differentials and residues on algebraic varieties.
\newblock {\em Ast\'{e}risque}, (117):ii+138, 1984.

\bibitem[M2]{M2}
D.~R. Grayson and M.~E. Stillman.
\newblock Macaulay 2, a software system for research in algebraic geometry,
  2006.
\newblock Available at \url{http://www.math.uiuc.edu/Macaulay2/}.


\bibitem[Mat89]{MatsCRT89}
H.~Matsumura.
\newblock {\em Commutative ring theory}, volume~8 of {\em Cambridge Studies in
  Advanced Mathematics}.
\newblock Cambridge University Press, Cambridge, second edition, 1989.
\newblock Translated from the Japanese by M. Reid.

\bibitem[ML63]{MacLHomology63}
S.~Mac~Lane.
\newblock {\em Homology}.
\newblock Die Grundlehren der mathematischen Wissenschaften, Band 114.
  Springer-Verlag, Berlin-G\"{o}ttingen-Heidelberg; Academic Press, Inc.,
  Publishers, New York, 1963.

\bibitem[Ses07]{SeshadriSMT07}
C.~S. Seshadri.
\newblock {\em Introduction to the theory of standard monomials}, volume~46 of
  {\em Texts and Readings in Mathematics}.
\newblock Hindustan Book Agency, New Delhi, 2007.
\newblock With notes by Peter Littelmann and Pradeep Shukla, Appendix A by V.
  Lakshmibai, Revised reprint of lectures published in the Brandeis Lecture
  Notes series.

\end{thebibliography}
\end{document}